\documentclass[12pt]{amsart}
\usepackage{amsmath,amsthm,amssymb,IMjournal}
\usepackage{graphicx}
\usepackage{caption}
\newtheorem{theorem}{Theorem}[section]
\newtheorem{acknowledgement}[theorem]{Acknowledgements}

\newtheorem{corollary}[theorem]{Corollary}

\newtheorem{definition}[theorem]{Definition}

\newtheorem{lemma}[theorem]{Lemma}

\newtheorem{proposition}[theorem]{Proposition}
\newtheorem{remark}[theorem]{Remark}

\usepackage{fullpage}
\begin{document}
\newtheorem{The}{Theorem}[section]

\numberwithin{equation}{section}

\title{QUADRATIC\ DIFFERENTIALS $A\left( z-a\right) \left( z-b\right)
dz^{2}/(z-c)^{2}\ $ AND ALGEBRAIC CAUCHY TRANSFORM}

\author{\|Mohamed Jalel |Atia|, Gabes.Tunisia,
        \\ \|Faouzi |Thabet|, Gabes.Tunisia}



\abstract  In this paper, we discuss the
representability almost everywhere (a.e.) in $\mathbb{C}$ of an irreducible algebraic function as the Cauchy transform of a signed
measure supported on a finite number of compact semi-analytic curves and a
finite number of isolated points. We discuss the existence of critical
trajectories of a family of quadratic differentials $\displaystyle\frac{%
A\left( z-a\right) \left( z-b\right) }{(z-c)^{2}}dz^{2}$.
\endabstract

\keywords
  Algebraic equation.\ Cauchy transform. Quadratic differentials.
\endkeywords


\section{Introduction}
We remind the problem set by B. Shapiro \cite{shapiro}: Is it true that if
there exists a signed measure whose Cauchy transform satisfies an
irreducible algebraic equation a.e. in $%
\mathbb{C}
$ then there exists, in general, another signed measure whose Cauchy
transform satisfies a.e. in $%
\mathbb{C}
$ the same algebraic equation and whose support is a finite union of compact
curves and isolated points? Does there exist such a measure with singularity
on each connected component of its support?

\bigskip

The aim of this paper is to solve the above problem in the case of algebraic
equation

\begin{equation}
p\left( z\right) h^{2}\left( z\right) -q\left( z\right) h\left( z\right)
+r=0,  \label{eq alg}
\end{equation}%
where $p$ and $q$ are polynomials of degree $1,$ and $r\in 
\mathbb{C}
^{\ast }$.

More precisely, we will investigate the existence of a compactly supported
positive measure whose Cauchy transform coincides with (a branch of) an
analytic continuation of a solution $h\left( z\right) $ of equation (\ref{eq
alg}) a.e. in $%
\mathbb{C}
.$ If such a real measure exists and its support is a finite union of
compact semi-analytic curves and isolated points we will call it a \emph{%
real motherbody measure} of (\ref{eq alg}). Recall that the \emph{Cauchy
transform} $\mathcal{C}_{\mu }$ of a compactly supported finite
complex-valued Borel measure $\mu $ is the analytic function defined by 
\begin{equation*}
\mathcal{C}_{\mu }\left( z\right) =\int_{%
\mathbb{C}
}\frac{d\mu \left( t\right) }{z-t},\quad z\in 
\mathbb{C}
\backslash \text{supp}\left( \mu \right) .
\end{equation*}%
For instance, if $P$ is a polynomial of degree $n,$ then the Cauchy
transform $\mathcal{C}_{P}$ of its normalized root-counting measure $%
\displaystyle\frac{1}{n}\sum_{p\left( a\right) =0}\delta _{a}$ where $\delta
_{a}$ is the Dirac measure supported at $a,$ is given by the formula 
\begin{equation*}
\mathcal{C}_{p}\left( z\right) =\frac{P^{\prime }\left( z\right) }{nP\left(
z\right) }=\sum_{p\left( a\right) =0}\frac{1}{z-a}.
\end{equation*}%
The Cauchy transform $\mathcal{C}_{\mu }\left( z\right) $ satisfies the
properties 
\begin{equation*}
\mathcal{C}_{\mu }\left( z\right) \sim \frac{\mu \left( 
\mathbb{C}
\right) }{z},\quad z\longrightarrow \infty ,\quad \mu =\frac{1}{\pi }\frac{%
\partial \mathcal{C}_{\mu }}{\partial \overline{z}}.
\end{equation*}%
\noindent A special case of equation (\ref{eq alg}) is 
\begin{equation}
zh^{2}\left( z\right) +\left( -z+A\right) h\left( z\right) +1=0,
\label{laguerre}
\end{equation}%
which appears in the study of the normalized root-counting measure $\mu
_{n}, $ 
\begin{equation*}
\mu _{n}=\mu \left( p_{n}\right) =\displaystyle\frac{\displaystyle%
\sum_{p_{n}\left( z\right) =0}\delta _{z}}{n}
\end{equation*}%
of the rescaled generalized Laguerre polynomials with varying parameters $nA$%
: 
\begin{equation*}
p_{n}\left( z\right) =L_{n}^{\alpha _{n}}\left( nz\right)
=\sum_{k=0}^{n}\left( 
\begin{array}{c}
n+nA \\ 
n-k%
\end{array}%
\right) \frac{\left( -z\right) ^{k}}{k!},
\end{equation*}%
with $A<-1$ in \cite{Kuijlaars}, and $A\notin 
\mathbb{R}
$ in \cite{paper1}. It is shown in \cite{weak} that the Cauchy transform of
the weak limit $\mu $ of $\mu _{n}$ satisfies equation (\ref{laguerre}), and
the support of the measure $\mu $ consists of the trajectories of a certain
quadratic differential connecting the zeros $a,b=A+2\pm 2\sqrt{A+1}$ of the
discriminant of equation (\ref{laguerre}).

Solutions of equation (\ref{eq alg}) are given by

\begin{equation*}
h\left( z\right) =\frac{q\left( z\right) -\sqrt{D\left( z\right) }}{2p\left(
z\right) },
\end{equation*}%
with some branch cut of the square root of the discriminant 
\begin{equation*}
D\left( z\right) =q^{2}\left( z\right) -4rp\left( z\right) =A\left(
z-a\right) \left( z-b\right) ,A\in 
\mathbb{C}
^{\ast },\left( a,b\right) \in 
\mathbb{C}
^{2}.
\end{equation*}

It is obvious that with the choice of the square root of $D$ with condition 
\begin{equation*}
\sqrt{D\left( z\right) }\sim q\left( z\right) ,z\rightarrow \infty ,
\end{equation*}%
there exists $\alpha \in 
\mathbb{C}
$ such that $h\left( z\right) \sim \displaystyle\frac{\alpha }{z}%
,z\rightarrow \infty .$

We begin our study by giving the following necessary conditions for the
existence of the real motherbody measure.

\begin{proposition}
If equation (\ref{eq alg}) admits a real motherbody measure $\mu $, then:

\begin{itemize}
\item any connected curve in the support of $\mu $ coincides with a
horizontal trajectory of the quadratic differential 
\begin{equation*}
\varpi =-\frac{D\left( z\right) }{p^{2}\left( z\right) }dz^{2}.
\end{equation*}

\item the support of $\mu $ should include both branching points of (\ref{eq
alg}) i.e. the zeros of $D.$
\end{itemize}
\end{proposition}

\proof
See e.g.\cite{shapiro} or \cite{pritsker}.
\endproof

\bigskip

Proposition 1, connects the motherbody measure with horizontal trajectories
of a quadratic differential. Quadratic differentials appear in many areas of
mathematics and mathematical physics such as orthogonal polynomials, moduli
spaces of algebraic curves, univalent functions, asymptotic theory of linear
ordinary differential equations etc...

Let us discuss some properties of horizontal trajectories of the rational
quadratic differential $\varpi =-\displaystyle\frac{D\left( z\right) }{%
p^{2}\left( z\right) }dz^{2}$ on the Riemann sphere $\hat{%
\mathbb{C}%
}.$

Zeros and simple poles of $-\displaystyle\frac{D\left( z\right) }{%
p^{2}\left( z\right) }dz^{2}$ are called \emph{finite critical points},
poles of order greater than two are called \emph{infinite critical points. }%
All other points are called \emph{regular points}.

The \emph{horizontal trajectories} (or just trajectories) of the quadratic
differential $\varpi $ are given by the equation 
\begin{equation}
\mathcal{\Re }\int^{z}\frac{\sqrt{D\left( t\right) }}{p\left( t\right) }%
\,dt\equiv const.  \label{re}
\end{equation}%
The \emph{vertical} or \emph{orthogonal} trajectories are obtained by
replacing $\Re $ by $\Im $ in the equation above.

The local structure of the trajectories is well known (see e.g.\cite%
{Strebel:84}~,\cite{MR0096806}, \cite{Pommerenke75}, \cite{MR1929066}). At
any regular point, the trajectory passing through this point is a close
analytic arc. Through every regular point of $\varpi $ passes uniquely
determined horizontal and vertical trajectories, which are orthogonal to
each other \cite[Theorem 5.5]{Strebel:84}. At a zero of multiplicity $r,$
there emanate $r+2$ trajectories under equal angles $\pi /\left( r+2\right) .
$ At a simple pole there emanates only one trajectory. At a double pole, the
local behaviour of the trajectories depends on the vanishing of the real or
imaginary part of the residue; they have either the radial, the circular or
the log-spiral form, see Figures \ref{Fig1},\ref{Fig2}.

\begin{figure}[h]
\begin{center}
\includegraphics[width=10cm,height=6cm]{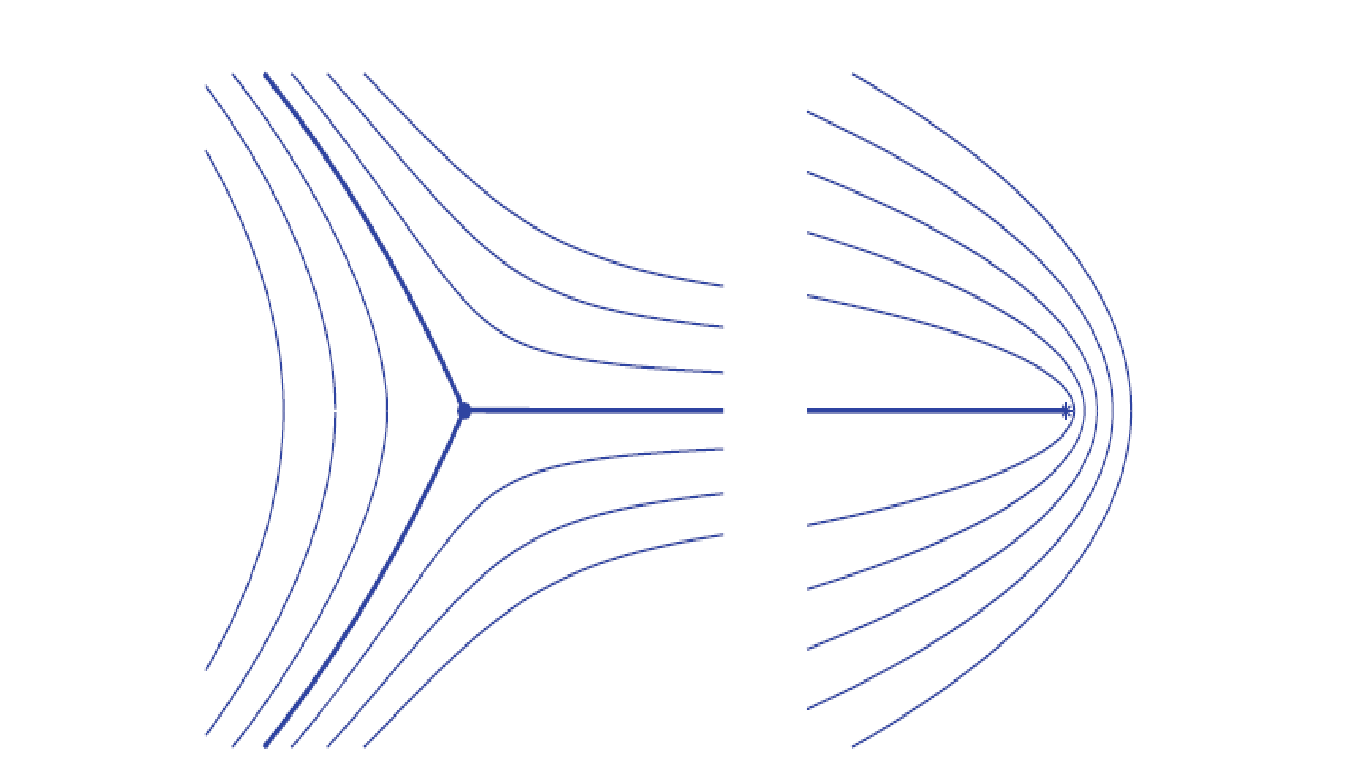}
\caption[le titre]{ The local trajectory structure near a simple zero (left) or a simple pole (right)
}
\label{Fig1}
\end{center}
\end{figure}
\begin{figure}[h]
\begin{center}
\includegraphics[width=10cm,height=6cm]{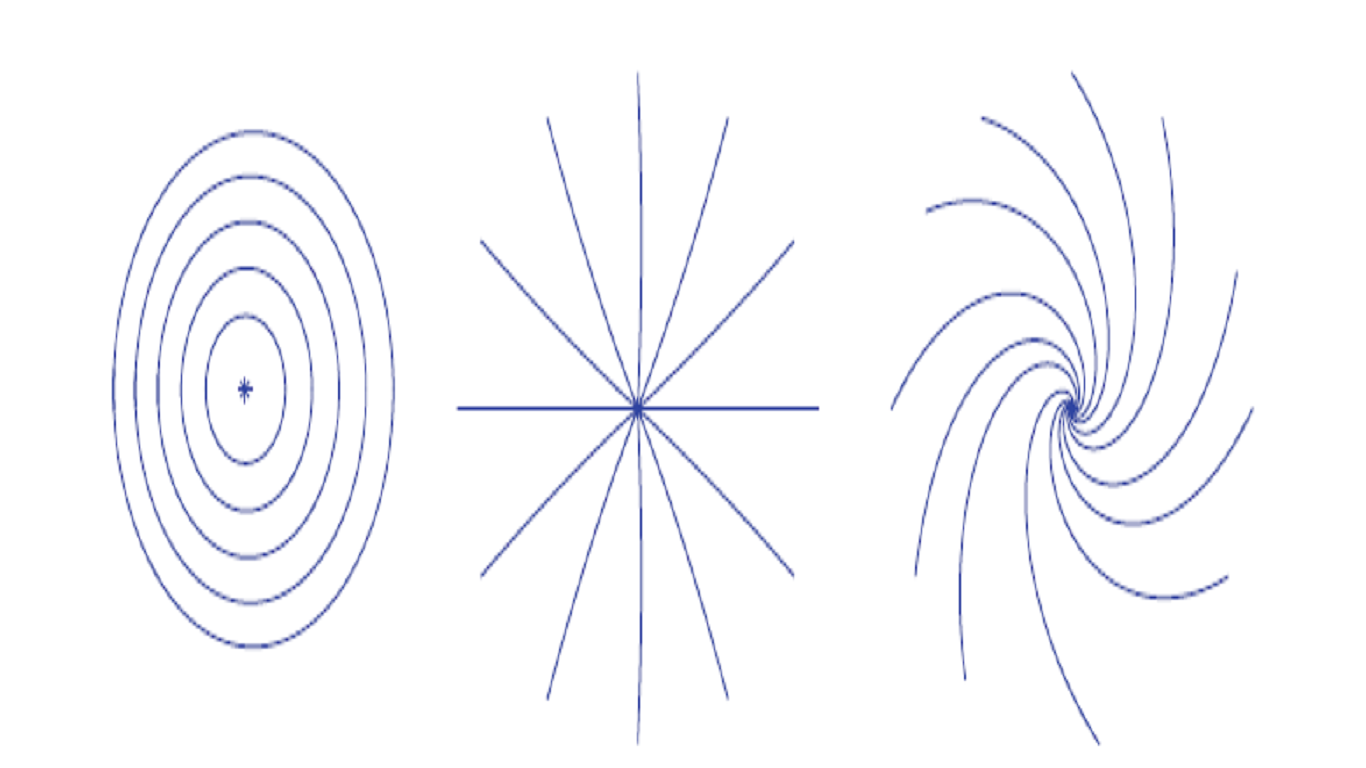}
\caption[le titre]{ The local behaviour of the trajectories near the origin, $ab > 0$ (left), $ab < 0$ (centre), and $ab \notin  \mathbb{R}$, (right).}
\label{Fig2}
\end{center}
\end{figure}
A trajectory of $\varpi $ starting and ending at finite critical points is
called \emph{finite critical} or \emph{short}. If it starts at a finite
critical point but tends either to the origin or to infinity, we call it 
\emph{an infinite critical trajectory} of $\varpi $.

The set of finite and infinite critical trajectories of $\varpi $ together
with their limit points (critical points of $\varpi $) is called the \emph{%
critical graph} of $\varpi $.

By a translation of the variable $z$ and the change of variable $\sqrt{A}%
z=y, $ we may assume, without loss of generality that, 
\begin{equation*}
\varpi =\varpi \left( z,a,b\right) =-\frac{\left( z-a\right) \left(
z-b\right) }{z^{2}}dz^{2},\ (a,b)\in 
\mathbb{C}
^{2}-\{(0,0)\}.
\end{equation*}%
We start by observing that $\varpi $ has two zeros, $a$ and $b$, and, if $%
ab\neq 0,$ the origin is a double pole, with 
\begin{equation*}
\ \varpi =\left( -\frac{ab}{z^{2}}+\mathcal{O}(z^{-1})\right) dz^{2},\quad
z\rightarrow 0.
\end{equation*}%
Another pole of $\varpi $ is located at infinity and is of order 4. In fact,
with the parametrization $u=1/z$, we get 
\begin{equation*}
\ \varpi =\left( -\frac{1}{u^{4}}+\mathcal{O}(u^{-3})\right) du^{2},\quad
u\rightarrow 0.
\end{equation*}%
If $a=0$ or $b=0$, the origin is a simple pole.\newline

Regarding the behavior at infinity, we can assume that the imaginary (resp.
real) axis is the only asymptotic direction of the trajectories (resp.
orthogonal trajectories) of $\varpi $. In other words, there exists a
neighborhood of infinity $U$ such that every trajectory entering $U$ tends
to $\infty $ either in the $+i\infty $ or $-i\infty $ direction, and the two
rays of any trajectory which stays in $U$ tend to $\infty $ in the opposite
asymptotic directions (\cite[Theorem 7.4]{Strebel:84}).

Usually, the main troubles in the description of the global structure of the
trajectories of a quadratic differential come from the existence of the
so-called recurrent trajectories, whose closures may have a non-zero plane
Lebesgue measure. However, since $\varpi $ has only two poles ( $0$ and $%
\infty $ ), Jenkins' Three Pole Theorem asserts that it cannot have any
recurrent trajectory (see \cite[Theorem 15.2]{Strebel:84}).\bigskip

\begin{itemize}
\item If $a=b$ then $\varpi =-\frac{\left( z-a\right) ^{2}}{z^{2}}dz^{2},$
and then there are $4$ trajectories emanating from $a$ under equal angles $%
\pi /2$,

\begin{itemize}
\item if $a\in 
\mathbb{R}
,$ two of them diverge to infinity parallel to the imaginary axis in
opposite directions; the two others, form a loop around the origin. In case $%
a=1$, we get the well-known Szeg\H{o} curve, see Figure \ref{Fig3}.
\begin{figure}[h]
\begin{center}
\includegraphics[width=10cm,height=6cm]{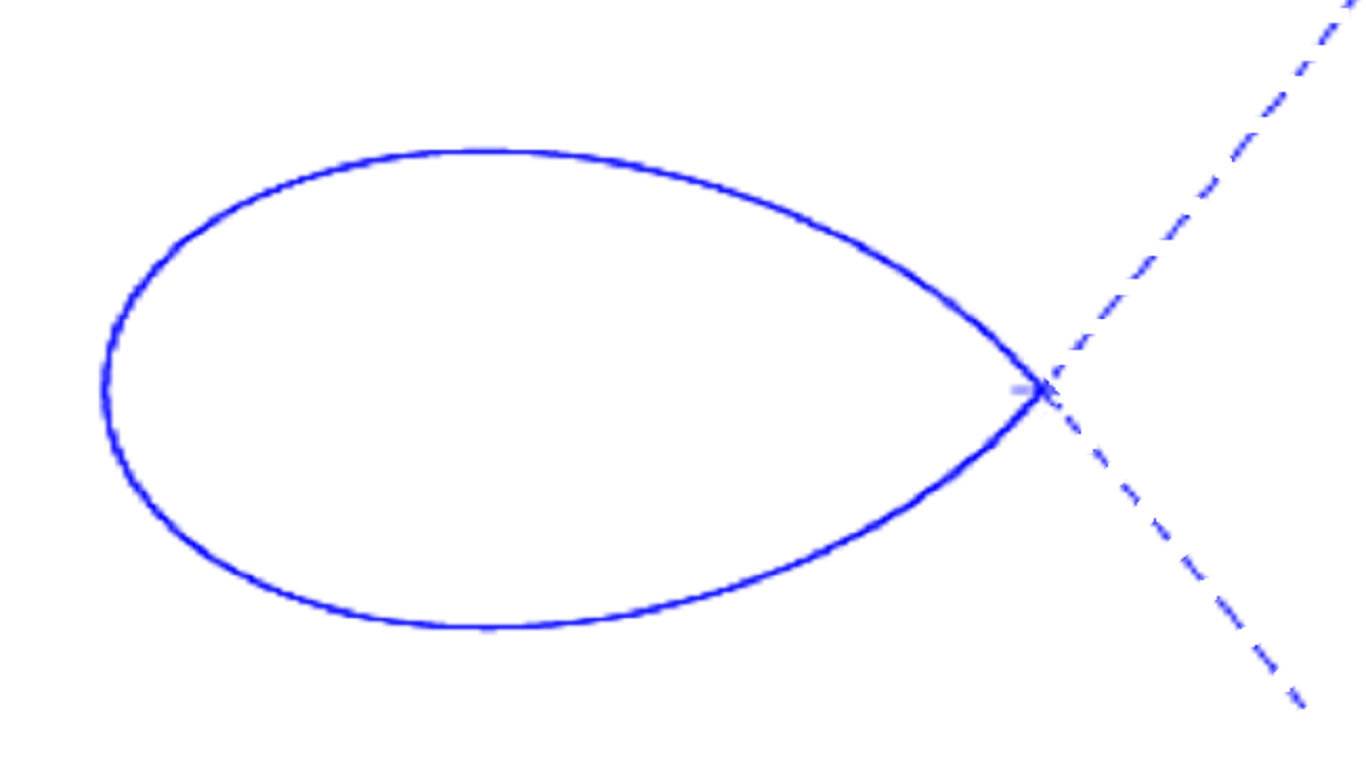}
\caption[le titre]{  Critical graph for the case $a=b=1$; and the Szeg\H{o} curve (solid line)
}
\label{Fig3}
\end{center}
\end{figure}
\item if $a\in i%
\mathbb{R}
,$ then, then the critical graph is composed by one of the sets $\left\{
iy;y\in 
\mathbb{R}
^{+}\right\} $ or $\left\{ iy;y\in 
\mathbb{R}
^{-}\right\} $ that contains $a,$ and the two other trajectories diverge to
infinity and form with infinity a domain that contains $0$.

\item if $a\notin 
\mathbb{R}
\cup i%
\mathbb{R}
,$ then one spiral trajectory diverges to the origin in , two trajectories
diverge to infinity in the same direction and form with infinity a domain
that contains the spiral, the fourth trajectory diverges to infinity in the
other direction.
\end{itemize}

\item If $a=0$ and $b\neq 0,$ then $\varpi =-\frac{z-b}{z}dz^{2},$ and there
emanate $3$ trajectories from $b$ under equal angles $2\pi /3$, one of them
goes to the origin, the two others go to infinity parallel to the imaginary
axis and in the opposite directions, see Figure \ref{Fig4}.
\begin{figure}[h]
\begin{center}
\includegraphics[width=10cm,height=6cm]{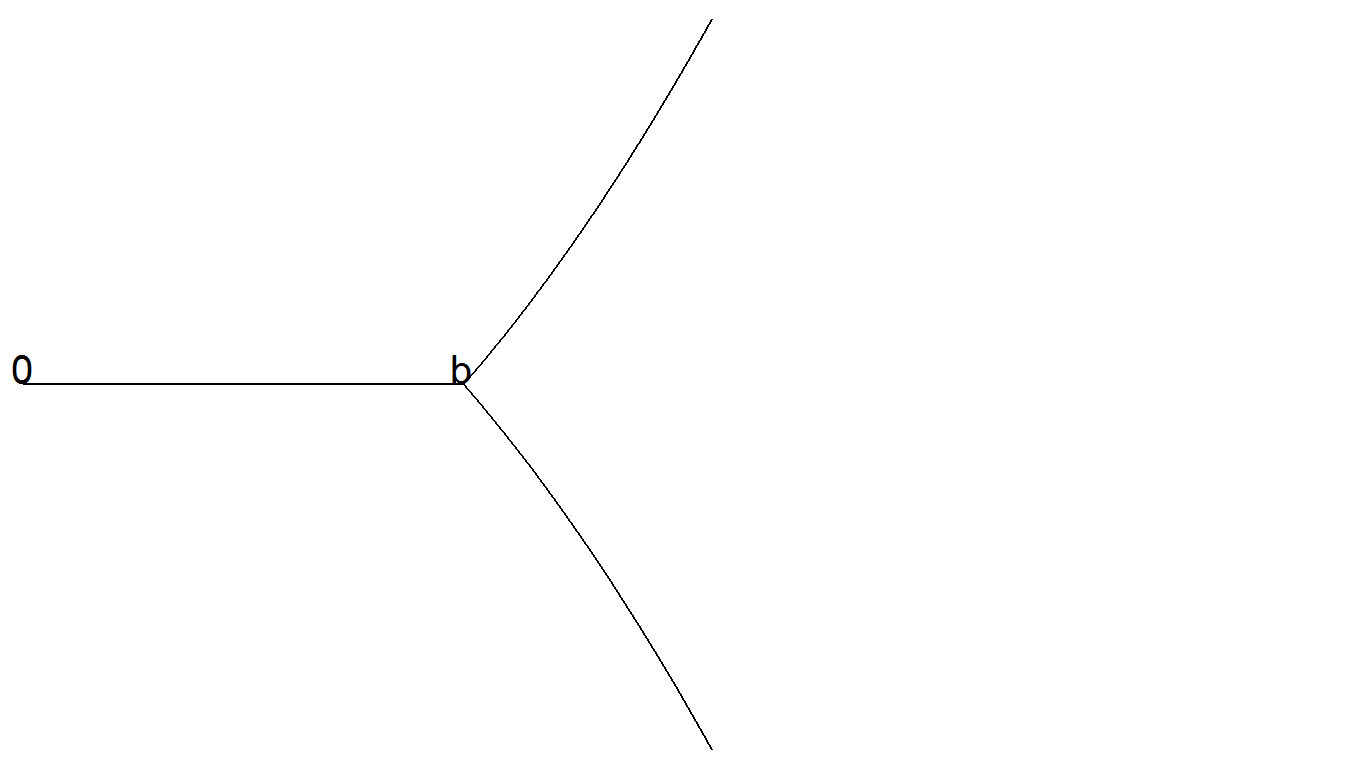}
\caption[le titre]{ Critical graph for the case $a = 0; b > 0.$}
\label{Fig4}
\end{center}
\end{figure}
\end{itemize}

In what follows, we investigate $a\neq b,$ and $ab\neq 0.$ In this case,
from each zero, $a$ and $b,$ there emanate $3$ trajectories under equal
angles $2\pi /3$. The local behavior of the trajectories near the origin
depends on the vanishing of the real or the imaginary part of the product $%
ab $.

\bigskip

The main result of this paper is the following.

\begin{proposition}
\label{main} Let $a$ and $b$ be two non vanishing complex numbers. There
exists a critical trajectory of the quadratic differential $\varpi \left(
z,a,b\right)$ if and only if $\left( \sqrt{a}+\sqrt{b}\right) ^{2}\in 
\mathbb{R}
$ or $\left( \sqrt{a}-\sqrt{b}\right) ^{2}\in 
\mathbb{R}
.$
\end{proposition}

\begin{remark}
\begin{enumerate}
\item[(i)] Results of Proposition 2 hold in the case of the rescaled
generalized Laguerre polynomials with varying parameters $nA$. In this case
we have $a,b=A+2\pm 2\sqrt{A+1}$ and 
\begin{equation*}
\left( \sqrt{a}\pm \sqrt{b}\right) ^{2}=2A+4\pm 2\sqrt{((A+2)^{2}-4(A+1)}%
=2A+4\pm 2\sqrt{A^{2}}.
\end{equation*}%
In other words $\left( \sqrt{a}\pm \sqrt{b}\right) ^{2}$ equals either $4A+4$
or $4$.\newline

\item[(ii)] Given a complex number $a,$ we consider the set%
\begin{equation*}
\Gamma _{a}=\left\{ b\in 
\mathbb{C}
\mid \left( \sqrt{a}+\sqrt{b}\right) ^{2}\in 
\mathbb{R}
\text{ or }\left( \sqrt{a}-\sqrt{b}\right) ^{2}\in 
\mathbb{R}
\right\} .
\end{equation*}%
Straightforward calculations show that if $a\notin 
\mathbb{R}
,$ then $\Gamma _{a}=\mathcal{P}_{1}\cup \mathcal{P}_{2},$ where $\mathcal{P}%
_{1}$ and $\mathcal{P}_{2}$ are the parabolas defined by : 
\begin{eqnarray*}
\mathcal{P}_{1} & = &\left\{ (x,y) \in \mathbb{R}^{2}\mid \Re(a) +2\left( \frac{y-\Im\left( a\right) 
}{2\Im\left( \sqrt{a}\right) }\right) \Re\left( \sqrt{a}\right)
+\left( \frac{y-\Im\left( a\right) }{2\Im\left( \sqrt{a}\right) }
\right) ^{2}=x \right \} , \\
\mathcal{P}_{2} &=&\left\{ \left( x,y\right) \in \mathbb{R}^{2}\mid \Re\left( a\right) -2\left( \frac{y-\Im\left( a\right) 
}{2\Re\left( \sqrt{a}\right) }\right) \Re\left( \sqrt{a}\right)
-\left( \frac{y-\Im\left( a\right) }{2\Im\left( \sqrt{a}\right) }
\right) ^{2}=x\right\} .
\end{eqnarray*}
\end{enumerate}

If $a>0,$ then $\Gamma _{a}=%
\mathbb{R}
^{+}\cup \left\{ \left( x,y\right) \in 
\mathbb{R}
^{2}\mid x=a-\frac{y^{2}}{4a}\right\} .$

If $a<0,$ then $\Gamma _{a}=%
\mathbb{R}
^{-}\cup \left\{ \left( x,y\right) \in 
\mathbb{R}
^{2}\mid x=a-\frac{y^{2}}{4a}\right\} .$
\begin{figure}[h]
\begin{center}
\includegraphics[width=10cm,height=6cm]{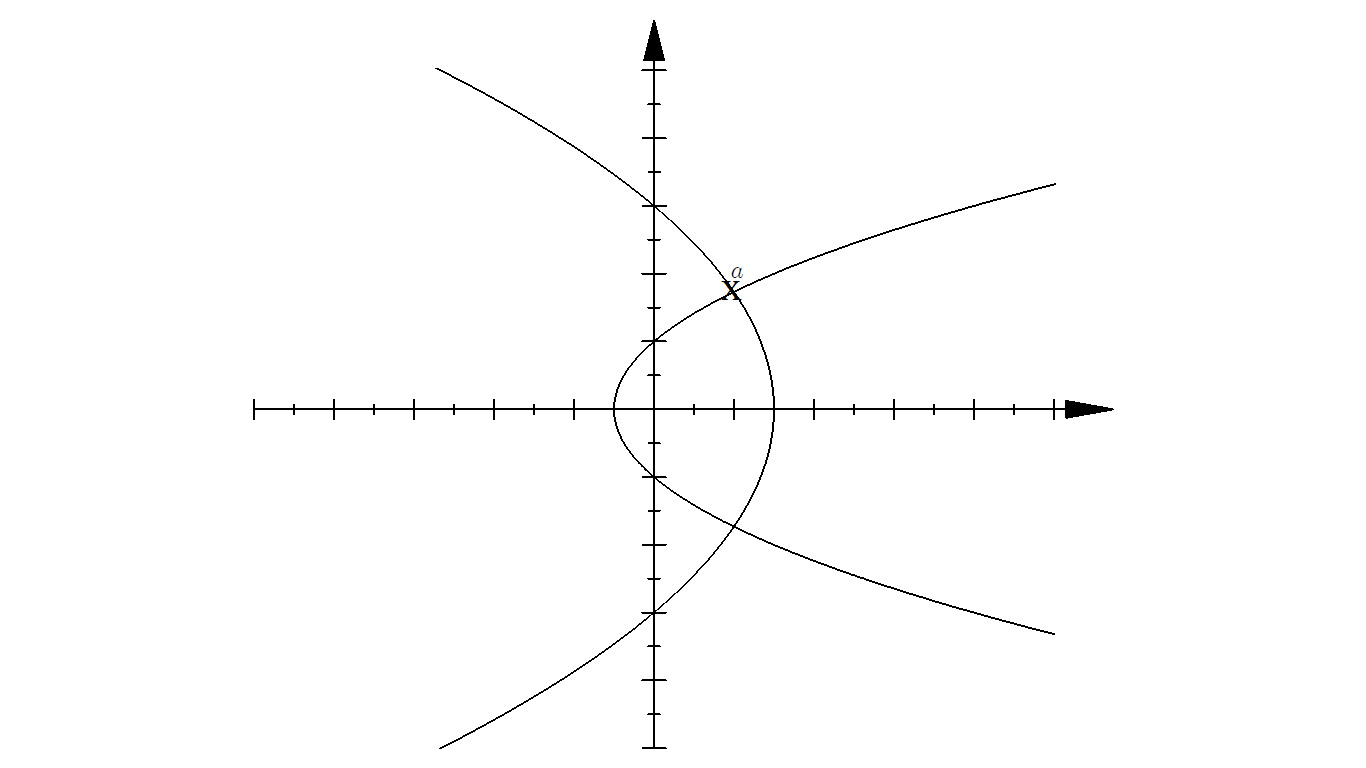}
\caption[le titre]{ Set  $\Gamma_a$  when $a\notin\mathbb{R}.$}
\label{Fig5}
\end{center}
\end{figure}
\end{remark}

\begin{acknowledgement}
\bigskip The authors are grateful to Professor Marco Bertola for useful
discussions, and they acknowledge the contribution of the anonymous referee
whose careful reading of the manuscript helped to improve the presentation.
\end{acknowledgement}

\section{Proof of Proposition 2.}

To prove Proposition 2. we need some lemmas. Below, given an oriented Jordan
curve $\Gamma $ joining $a$ and $b$ in $%
\mathbb{C}
^{\ast }$, for $t\in \Gamma $, we denote by $\sqrt{D(t)}_{+}$ and $\sqrt{D(t)%
}_{-}$ the limits from the $+$-side and $-$-side respectively. (As usual,
the $+$-side of an oriented curve lies to the left, and the $-$-side lies to
the right, if one traverses the curve according to its orientation.)

\begin{lemma}
\label{real} For any curve $\gamma $ joining $a$ and $b$ and not passing
through $0,$ we have :%
\begin{equation*}
\int_{\gamma }\dfrac{\left( \sqrt{\ D\left( z\right) }\right) _{+}}{z}dz=\pm 
\frac{i\pi }{2}\left( \sqrt{a}\pm \sqrt{b}\right) ^{2},
\end{equation*}%
the signs $\pm $ depend on the homotopy class of $\gamma $ in $%
\mathbb{C}
^{\ast },$ and the branch of the square root $\sqrt{\ D\left( z\right) }$
defined in $%
\mathbb{C}
\backslash \gamma $ is chosen so that $\sqrt{\ D\left( z\right) }\sim
z,z\rightarrow \infty .$
\end{lemma}

\begin{proof}
With the above choices, consider $I=\int_{\gamma }\frac{\left( \sqrt{\
D\left( z\right) }\right) _{+}}{z}dz.$ Since $\left( \frac{\sqrt{\ D\left(
t\right) }}{t}\right) _{+}=-\left( \frac{\sqrt{\ D\left( t\right) }}{t}%
\right) _{-}$ for $t\in \gamma ,$ we have 
\begin{equation*}
2I=\int_{\gamma }\left[ \left( \frac{\sqrt{\ D\left( t\right) }}{t}\right)
_{+}-\left( \dfrac{\sqrt{\ D\left( t\right) }}{t}\right) _{-}\right]
dt=\oint_{\Gamma }\dfrac{\sqrt{\ D\left( z\right) }}{z}dz,
\end{equation*}%
where $\Gamma $ is a closed contour encircling the curve $\gamma $ once in
the clockwise direction and not encircling $z=0.$ After a contour
deformation we pick up residues at $z=0$ and at $z=\infty $ for the
calculation of $I,$ namely :%
\begin{equation*}
2I=\pm 2i\pi \left( \underset{z=0}{Res}\left( \dfrac{\sqrt{\ D\left(
z\right) }}{z}\right) +\underset{z=\infty }{Res}\left( \dfrac{\sqrt{\
D\left( z\right) }}{z}\right) \right) .
\end{equation*}%
Clearly, 
\begin{equation*}
\underset{z=0}{Res}\left( \dfrac{\sqrt{\ D\left( z\right) }}{z}\right) =%
\sqrt{\ D\left( 0\right) }=\sqrt{\ ab}.
\end{equation*}%
The residue at $\infty $ is the opposite of the coefficient of $\frac{1}{z}$
in the Laurent serie of $\dfrac{\sqrt{\ D\left( z\right) }}{z}.$ Since $%
\frac{\sqrt{D\left( z\right) }}{z}\sim 1,z\rightarrow \infty ,$ we have%
\begin{equation*}
\dfrac{\sqrt{\ D\left( z\right) }}{z}=1-\frac{a+b}{2}\frac{1}{z}+\mathcal{O}%
\left( \frac{1}{z^{2}}\right) ;
\end{equation*}%
and therefore 
\begin{equation*}
\underset{z=\infty }{Res}\left( \dfrac{\sqrt{\ D\left( z\right) }}{z}\right)
=\dfrac{a+b}{2}.
\end{equation*}
\end{proof}

As an immediate consequence of Lemma \ref{real} we get

\begin{corollary}
\label{2spirals2zeros} If $\left( \sqrt{a}+\sqrt{b}\right) ^{2}\in 
\mathbb{R}
$ or $\left( \sqrt{a}-\sqrt{b}\right) ^{2}\in 
\mathbb{R}
,$ then, there cannot exist two horizontal trajectories emanating from $a$
and $b$ and diverging simultaneously to the origin.

Alternatively, if $\left( \sqrt{a}+\sqrt{b}\right) ^{2}\in i%
\mathbb{R}
$ or $\left( \sqrt{a}-\sqrt{b}\right) ^{2}\in i%
\mathbb{R}
,$ then, there cannot exist two vertical trajectories emanating from $a$ and 
$b$ and diverging simultaneously to the origin.
\end{corollary}

\begin{proof}
Assume that $ab\notin 
\mathbb{R}
,$ and let $\gamma _{a}$ and $\gamma _{b}$ be two trajectories that diverge
in spirals to the origin. Let $\sigma $ be an orthogonal trajectory that
diverges to the origin. Then $\sigma $ intersects $\gamma _{a}$ and $\gamma
_{b}$ infinitely many times. Considering three consecutive points of
intersection, it is obvious that we can construct two paths $\gamma ,\gamma
^{\prime }$ joining $a$ and $b$ and not homotopic in $%
\mathbb{C}
^{\ast },$ formed by the three pieces, from $\gamma _{a},\sigma $ and $%
\gamma _{b}.$ Then we get 
\begin{equation*}
\Re \int_{\gamma }\dfrac{\left( \sqrt{\ D\left( z\right) }\right) _{+}}{z}%
dz\neq 0,\text{ and }\Re \int_{\gamma ^{\prime }}\dfrac{\left( \sqrt{\
D\left( z\right) }\right) _{+}}{z}dz\neq 0,
\end{equation*}%
which contradicts Lemma \ref{real}.

If $ab<0,$ the loop formed by vertical trajectories passes through only one
zero, either $a$ or $b;$ we can repeat the same proof as in the previous
case.
\end{proof}

\begin{definition}
\bigskip A domain in $%
\mathbb{C}
$ bounded only by segments of horizontal and/or vertical trajectories of $%
\varpi $ (and their endpoints) is called $\varpi $-polygon.
\end{definition}

We can use the Teichm\"{u}ller lemma (see [4, Theorem 14.1]) to clarify some
facts about the global strucure of the trajectories.

\begin{lemma}[Teichmuller]
\label{teich} Let $\Omega $ be a $\varpi $-polygon, and let $z_{j}$ be the
singular points of $\varpi $ on the boundary $\partial \Omega $ of $\Omega ,$
with multiplicities $n_{j},$ and let $\theta _{j}$ be the corresponding
interior angles with vertices at $z_{j},$ respectively. Let 
\begin{equation*}
\beta _{j}=1-\theta _{j}\dfrac{n_{j}+2}{2\pi }.
\end{equation*}%
Then%
\begin{equation*}
\sum \beta _{j}=\left\{ 
\begin{array}{c}
0,\text{if }0\in \Omega , \\ 
\\ 
2,\text{if }0\notin \Omega .%
\end{array}%
\right.
\end{equation*}
\end{lemma}

Any $\varpi $-polygon made of horizontal trajectories and containing the
origin can be bounded either by two critical trajectories starting and
ending at $a,b,$ or it must contain $\infty $ at its boundary and at least
one inner angle $\dfrac{4\pi }{3}.$

\begin{corollary}
\label{1coroll teich}In the latter case there are a priori three
possibilities:

\begin{itemize}
\item either $\Omega $ is bounded by two critical arcs emanating from the
same zero of $\varpi $ and forming an angle $\dfrac{4\pi }{3},$ encircling
the origin and going to $\infty $ in the same direction, or

\item $\Omega $ is bounded by two critical arcs emanating from the same zero
of $\varpi $ and forming an angle $\dfrac{2\pi }{3},$ encircling the origin
and the other zero, and going to $\infty $ in the same direction, or

\item $\Omega $ is bounded by two critical arcs emanating from different
zeros of $\varpi $ and forming an angle $\dfrac{4\pi }{3}$ going to $\infty $
in the opposite directions.
\end{itemize}
\end{corollary}

\begin{corollary}
\label{2spirals1zero}There cannot exist two horizontal, or vertical
trajectories emanating from the same zero $a$ or $b$ and diverging (radially
or spirally) to the origin.
\end{corollary}

\begin{proof}
If there emanate two trajectories from $a$ or $b$ diverging to the origin,
consider an $\varpi $-polygon formed by their pieces and a piece of an
orthogonal trajectory that diverges to the origin. Clearly this $\varpi $%
-polygon violates Lemma \ref{teich}.
\end{proof}

\begin{corollary}
\label{1spiral2spirals} Assume that there is no critical trajectory of $%
\varpi $, then:\newline
if ($ab\notin 
\mathbb{R}
,$ or $ab<0$), we get,

\begin{itemize}
\item either, there exists one trajectory diverging to the origin, four
trajectories diverge to infinity in the same direction, and one trajectory
diverges to infinity in the other direction.

\item or, from each zero $a$ and $b,$ there emanates one trajectory
diverging to the origin, and two trajectories diverge to infinity in the
opposite directions.
\end{itemize}

\noindent If $ab>0,$ then, from one zero, there is a loop encircling the
origin, the third trajectory diverges to infinity. All trajectories
emanating from the other zero diverge to infinity, two of them, in the same
direction and form with infinity a domain that contains the loop, the third
one diverges to infinity in the opposite direction, see Figure \ref{Fig6}.
\begin{figure}[h]
\begin{center}
\includegraphics[width=10cm,height=6cm]{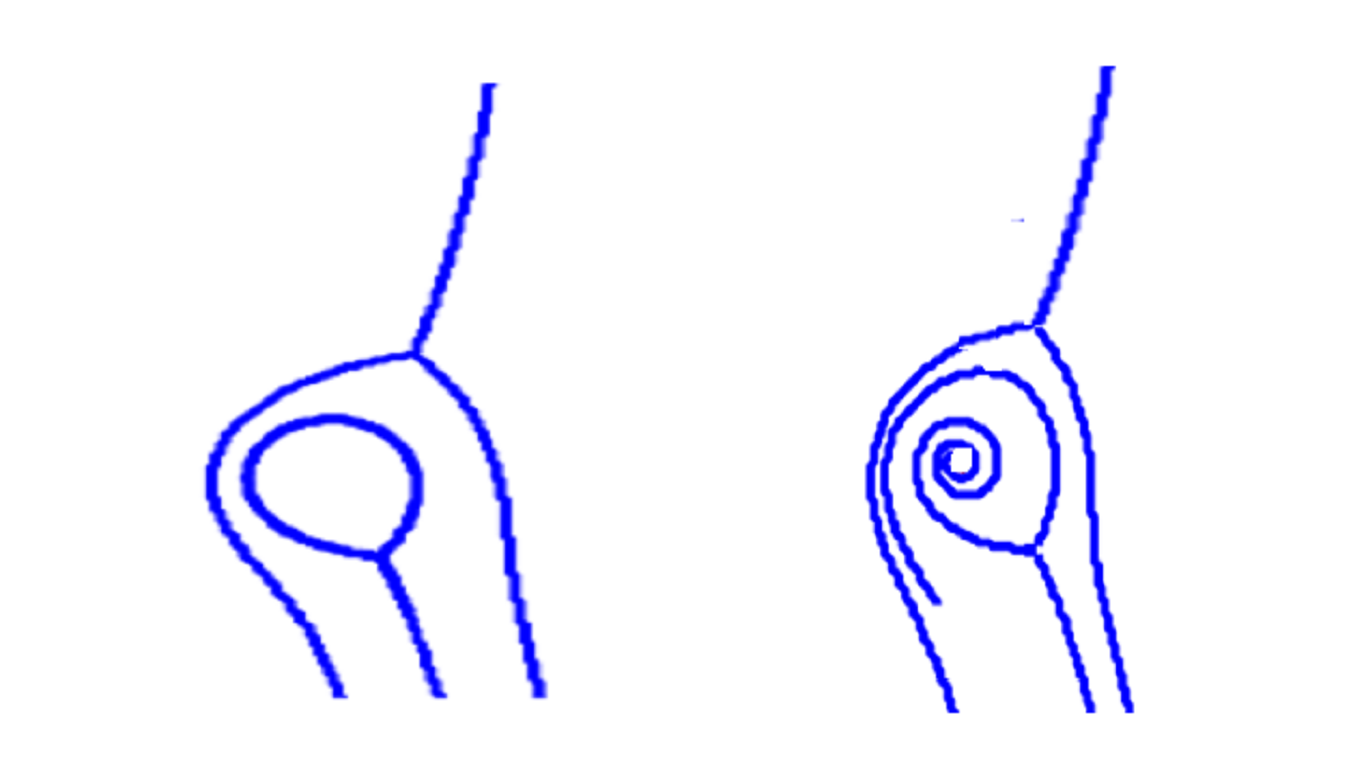}
\caption[le titre]{ Critical graph when $(\sqrt{a}\pm \sqrt{b})^2\notin\mathbb{R}$, $ab > 0$ (left), $ab \notin \mathbb{R}$ (right).}
\label{Fig6}
\end{center}
\end{figure}
\end{corollary}

\begin{proof}
If $\left( \sqrt{a}\pm \sqrt{b}\right) ^{2}\notin 
\mathbb{R}
$ and $ab\notin 
\mathbb{R}
$, then, by Lemma \ref{real} there is no critical trajectory, and by
Corollary \ref{2spirals1zero}, from each zero, there emanates at most one
trajectory that diverges to the origin. Suppose that all trajectories
emanating from a zero (for instance $a$ ) diverge to infinity, then, by
Lemma \ref{teich}, two of them, say $\gamma _{1},\gamma _{2}$ diverge in the
same direction and form, with infinity, a domain $\mathcal{D}$ which
contains the origin. By Lemma \ref{teich}, the third trajectory emanating
from $a,$ cannot diverge to infinity in the same direction as $\gamma
_{1},\gamma _{2}$. Corollary \ref{1coroll teich} implies that the interior
angle of $\mathcal{D}$ between $\gamma _{1}$ and $\gamma _{2}$ equals $\frac{%
2\pi }{3}.$ Then, the domain $\mathcal{D}$ must contain the origin and the
other zero $b.$\newline
All these considerations show that there emanate two trajectories from $b$
with the angle $\frac{4\pi }{3}$ that diverge to infinity in the direction
of $\gamma _{1},\gamma _{2},$ and which form, with infinity, a domain
containing the origin. The third trajectory emanating from $b$ diverges to
the origin.\newline
The remaining cases are settled in a similar way
\end{proof}

\begin{proof}[Proof of Propostion 2.]
\begin{enumerate}
\item If $\left( \sqrt{a}\pm \sqrt{b}\right) ^{2}\in 
\mathbb{R}
$, then, $ab>0$ and we have a loop around $0$. By a change of variable $z=%
\overline{y},$ we see that the critical graph of $\varpi $ is symmetric with
respect to the real axis. Thus it follows that

\begin{itemize}
\item either, $a,b\in 
\mathbb{R}
$ such that $ab>0$ in which case, the segment $\left[ a,b\right] $ belongs
to the critical graph, and the loop passes through exactly one zero. We have
totally two critical trajectories.

\item or $a=\overline{b}$ and the loop passes through $a$ and $b$, and,
again, we have two critical trajectories \cite{Kuijlaars}, see Figure \ref{Fig7}.
 \begin{figure}[h]
\begin{center}
\includegraphics[width=10cm,height=6cm]{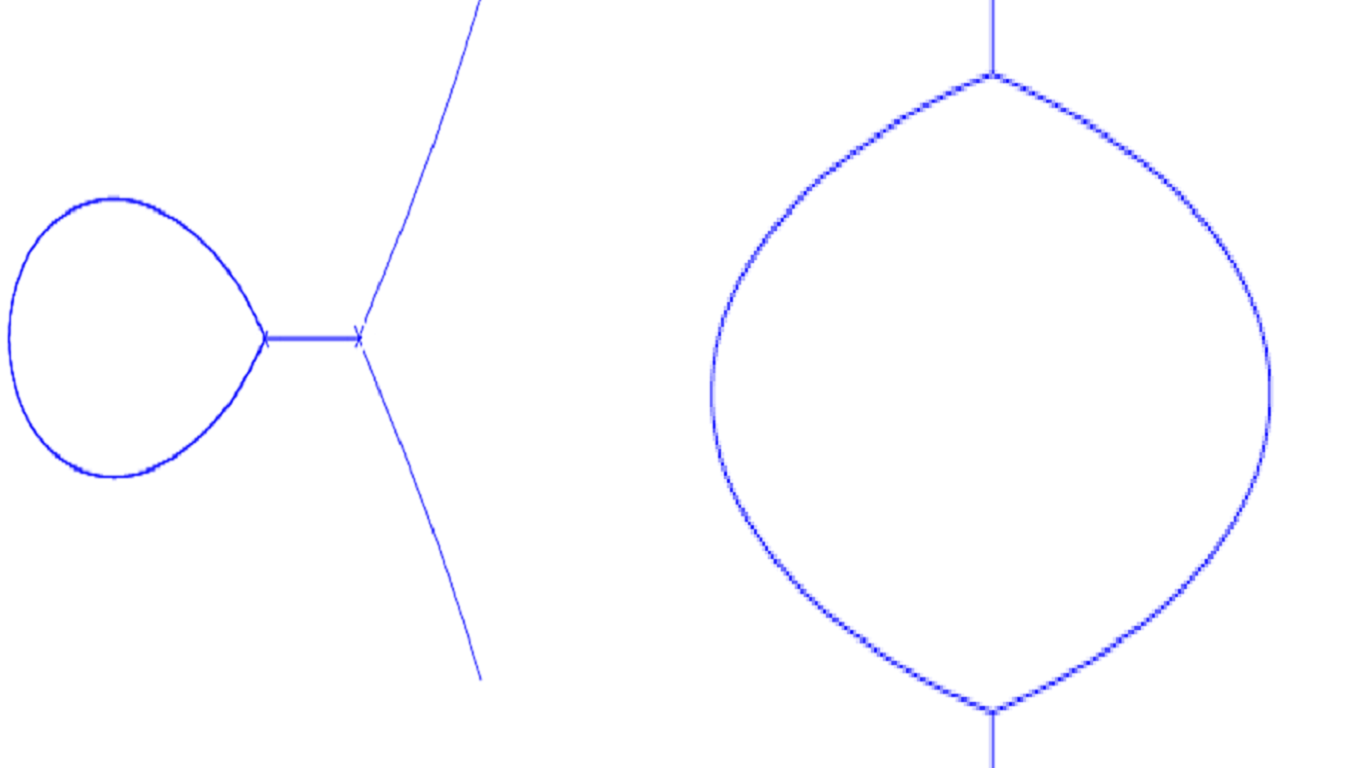}
\caption[le titre]{ Critical graph when $(\sqrt{a}\pm \sqrt{b})^2\in \mathbb{R}$, $a,b > 0$ (left), $b=\bar{a}$ (right).}
\label{Fig7}
\end{center}
\end{figure}
\end{itemize}

\item If $\left( \sqrt{a}+\sqrt{b}\right) ^{2}\in 
\mathbb{R}
$ and $\left( \sqrt{a}-\sqrt{b}\right) ^{2}\notin 
\mathbb{R}
$ then $ab\notin 
\mathbb{R}
$ or $ab<0.$\newline
Suppose that $ab\notin 
\mathbb{R}
$. By Corollaries \ref{2spirals2zeros}, \ref{2spirals1zero} and \ref%
{1spiral2spirals}, the critical graph of $\varpi $ possesses exactly one
trajectory that emanates from a zero, say $a$ and diverges to the origin.
From the other zero $b$, there emanate two trajectories, say $\gamma
_{1},\gamma _{2},$ diverging to infinity in the same direction and they
form, with infinity, a domain that contains the origin and $a$. From $a,$
there emanates at least one vertical trajectory $\sigma _{1}$ that diverges
to infinity (parallel to the real axis). This trajectory must intersects $%
\gamma _{1}$ or $\gamma _{2}$ at some point $M$. Let $\gamma $ be a path
connecting $a$ and $b$ in $%
\mathbb{C}
^{\ast },$ formed by a piece from $a$ to $M$ of $\sigma _{1}$ and another
one from $M$ to $b$ of $\gamma _{1}$ or $\gamma _{2}$. It follows that 
\begin{eqnarray}
\Re \left( \int_{\gamma }\dfrac{\sqrt{\ D\left( z\right) }}{z}dz\right) 
&=&\Re \left( \int_{a}^{M}\dfrac{\sqrt{\ D\left( z\right) }}{z}dz\right)
\neq 0,  \label{M and N} \\
\Im \left( \int_{\gamma }\dfrac{\sqrt{\ D\left( z\right) }}{z}dz\right) 
&=&\Im \left( \int_{M}^{b}\dfrac{\sqrt{\ D\left( z\right) }}{z}dz\right)
\neq 0.  \notag
\end{eqnarray}%
By Corollary \ref{2spirals1zero}, there emanates at least one vertical
trajectory $\sigma _{2}$ from $a$, which, either, connects $a$ to $b,$ or,
it diverges to infinity intersecting $\gamma _{1}$ or $\gamma _{2}$ at some
point $N.$ The first case contradicts equations (\ref{M and N}) and the fact
that $\left( \sqrt{a}+\sqrt{b}\right) ^{2}\in 
\mathbb{R}
.$ Let $\gamma ^{\prime }$ be a path connecting $a$ and $b$ in $%
\mathbb{C}
^{\ast }$ formed by a piece from $a$ to $N$ of $\sigma _{2}$ and another one
from $N$ to $b$ of $\gamma _{1}$ or $\gamma _{2}$. Clearly, by lemma \ref%
{teich}, $\gamma $ and $\gamma ^{\prime }$ are not homotopic in $%
\mathbb{C}
^{\ast }$ and equations (\ref{M and N}) are valid with $\gamma ^{\prime }.$
This contradicts the fact that $\left( \sqrt{a}+\sqrt{b}\right) ^{2}\in 
\mathbb{R}
.$ \newline
We conclude that a critical trajectory exists. The case $ab<0$ can be
settled in the same way. See Figure \ref{Fig8}.
\begin{figure}[h]
\begin{center}
\includegraphics[width=10cm,height=6cm]{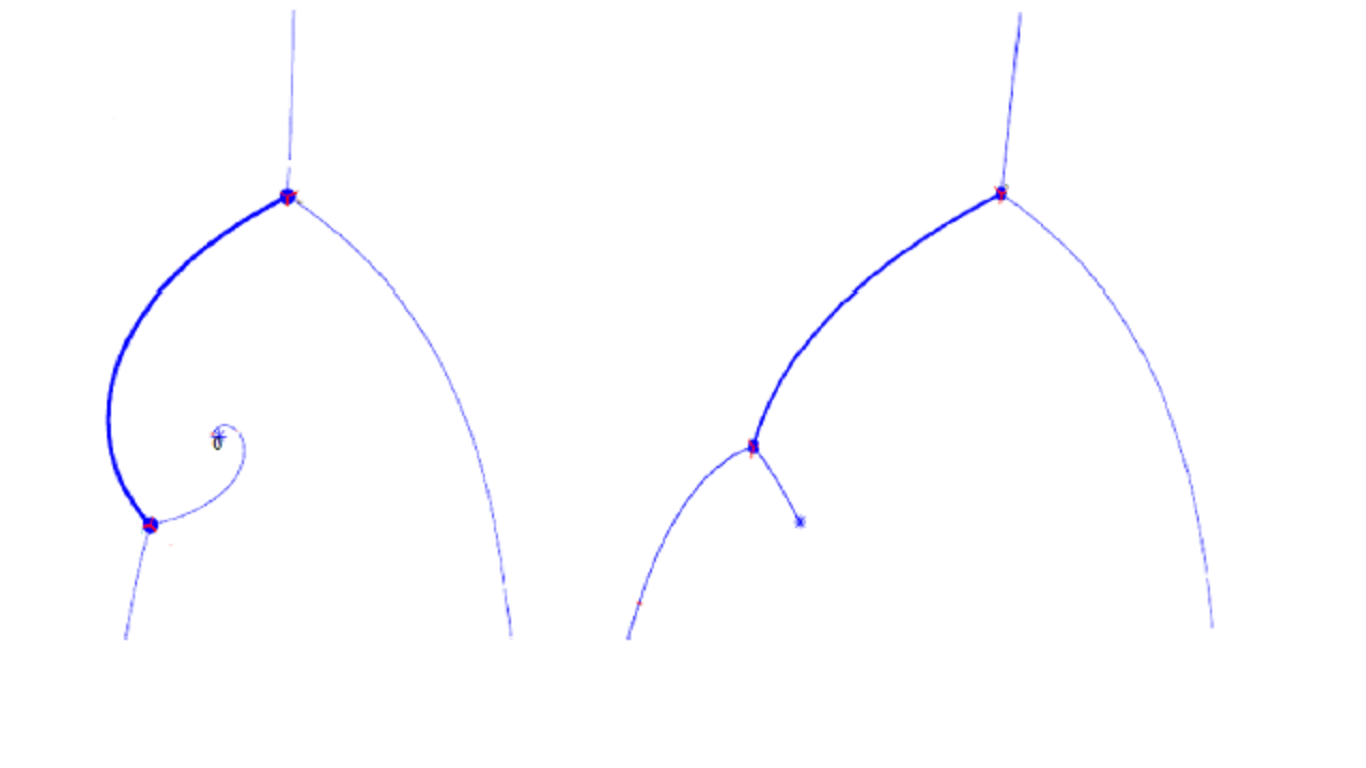}
\caption[le titre]{ Critical graph when $(\sqrt{a}+\sqrt{b})^2\in \mathbb{R}$ and $(\sqrt{a}-\sqrt{b})^2\notin\mathbb{R}$ , $ab \notin\mathbb{R}$ (left), $ab < 0$ (right).}
\label{Fig8}
\end{center}
\end{figure}
\item If $\left( \sqrt{a}\pm \sqrt{b}\right) ^{2}\notin 
\mathbb{R}
,$ then by Lemma \ref{real}, 
\begin{equation*}
\Re \left( \int_{a}^{b}\dfrac{\sqrt{\ D\left( z\right) }}{z}dz\right) \neq 0,
\end{equation*}%
for any path of integration in $%
\mathbb{C}
^{\ast }$ and there is no critical trajectory.
\end{enumerate}
\end{proof}

{\small
{\em Authors' addresses}:
\\{\em Mohamed Jalel Atia}, Facult\'e des sciences de Gab\`{e}s, Cit\'e Erriadh Zrig 6072 Gab\`{e}s Tunisia
 e-mail: \texttt{jalel.atia@\allowbreak gmail.com}.
\\{\em Faouzi Thabet}, ISSAT de Gab\`{e}s, Avenue Omar Ibn El Khattab 6029 Gab\`{e}s  Tunisia
 e-mail: \texttt{faouzithabet@\allowbreak yahoo.fr}.
}
\end{document}